\documentclass[11pt]{article}
\usepackage[numbers,sort&compress]{natbib}
\usepackage{enumerate}
\usepackage{tikz}
\usepackage{pgfplots}
\usepackage{amscd}
\usepackage{amsmath}
\usepackage{latexsym}
\usepackage{amsfonts}
\usepackage{amssymb}
\usepackage{amsthm}
\usepackage{verbatim}
\usepackage{mathrsfs}
\usepackage{enumerate}
\usepackage{hyperref}

\theoremstyle{plain}\newtheorem{definition}{Definition}[section]
\theoremstyle{definition}\newtheorem{theorem}{Theorem}[section]
\theoremstyle{plain}\newtheorem{lemma}[theorem]{Lemma}
\theoremstyle{plain}\newtheorem{coro}[theorem]{Corollary}
\theoremstyle{plain}\newtheorem{prop}[theorem]{Proposition}
\theoremstyle{remark}\newtheorem{remark}{Remark}[section]
\theoremstyle{definition}
\usepackage{xcolor}

\newcommand{\wred}[1]{\textcolor{black}{#1}}
\newcommand{\wwred}[1]{\textcolor{black}{#1}}

\newcommand{\norm}[1]{\left\|#1\right\|}

\newcommand{\B}{\Big}
\newcommand{\D}{\text{Div\,}}

\newcommand{\be}{\begin{equation}}
\newcommand{\ee}{\end{equation}}
\newcommand{\ba}{\begin{aligned}}
	\newcommand{\ea}{\end{aligned}}

\providecommand{\bysame}{\leavevmode\hbox to3em{\hrulefill}\thinspace}
\newcommand{\f}{\frac}

\newcommand{\ben}{\begin{enumerate}}
	\newcommand{\een}{\end{enumerate}}

\newcommand{\Rmnum}[1]{\expandafter\@slowromancap\romannumeral #1@}

\allowdisplaybreaks

\numberwithin{equation}{section}
\begin{document}
	\title{Fractal dimension of potential singular points set in the Navier-Stokes equations under supercritical    regularity }
	\author{Yanqing Wang\footnote{ Department of Mathematics and Information Science, Zhengzhou University of Light Industry, Zhengzhou, Henan  450002,  P. R. China Email: wangyanqing20056@gmail.com}\; ~~and~~~Gang Wu\footnote{Corresponding author. School of Mathematical Sciences,  University of Chinese Academy of Sciences, Beijing 100049, P. R. China Email: wugang2011@ucas.ac.cn}
 }
\date{}
\maketitle

\begin{abstract}
The main objective of this paper is  to answer the questions posed by Robinson and   Sadowski \cite[p. 505, Comm. Math. Phys., 2010]{[RS3]} for the Navier-Stokes equations.
Firstly,  we prove that
the upper box dimension of  the  potential singular points set $\mathcal{S}$ of   suitable weak solution $u$ belonging in $  L^{q}(0,T;L^{p}(\mathbb{R}^{3}))$ for $1\leq\f{2}{q}+\f{  3}{p}\leq\f32$ with $2\leq q<\infty$ and $2<p<\infty$ is at most  $\max\{p,q\}(\f{2}{q}+\f{  3}{p}-1)$ in this system. Secondly, it is shown that
$1-2 s$ dimension  Hausdorff measure  of potential singular points set of suitable  weak solutions
satisfying  $  u\in L^{2}(0,T;\dot{H}^{s+1}(\mathbb{R}^{3}))$ for $0\leq s\leq\f12$   is zero, whose proof relies on   Caffarelli-Silvestre's extension.  Inspired by Baker-Wang's recent work \cite{[BW]}, this further allows us to discuss the Hausdorff dimension of  potential singular points set of  suitable weak solutions if the gradient of the velocity under some supercritical    regularity.

  \end{abstract}
	\section{Introduction}
	\label{intro}
	\setcounter{section}{1}\setcounter{equation}{0}
We consider the  three-dimensional incompressible non-stationary Navier-Stokes equations
\be\label{NS}\left\{\ba
&u_{t}- \Delta u+u\cdot \nabla u+ \nabla \Pi =0, ~~\text{div}\, u=0~~\text{in}~~~~\mathbb{R}^{3}\times(0,\,T), \\
&u|_{t=0}=u_{0}(x)~~~~~~~~~\text{on} ~~~ \mathbb{R}^{3}\times\{t=0\}.
\ea\right.\ee
 Here $u$ describes the  velocity  of the flow and the
scalar function $\Pi$ represents the pressure of the fluid.
The  initial data   $u_{0}(x)$ satisfies divergence free condition.

The full regularity of solutions of  the three dimensional  Navier-Stokes equations
is not known, the partial regularity theory of suitable weak solutions  of this system  starting  from Scheffer's work \cite{[Scheffer1],[Scheffer2],[Scheffer3]} is well-known. The famous Caffarelli-Kohn-Nirenberg theorem in \cite{[CKN]}
 is that   one
dimensional Hausdorff measure of the potential space-time singular points set $\mathcal{S}$ of suitable
weak solutions to the 3D Navier-Stokes equations is zero.
The critical tool is the following   so-called $\epsilon$-regularity  criterion: there is an absolute constant $\epsilon$ such that if
   \be\label{ckn}
  \limsup_{\varrho\rightarrow0}\f{1}{\varrho}\iint_{Q(\varrho)}|\nabla u|^{2}dxdt
  \leq \epsilon,
  \ee
  then  $u$ is bounded in a neighborhood of $(0,0)$,   where
   $Q(\varrho):=B(\varrho)\times(-\varrho^{2},0)$ and $B(\varrho)$
   denotes the ball of center $0$ and radius $\varrho$. To this end, Caffarelli-Kohn-Nirenberg  \cite{[CKN]}  established $\epsilon$-regularity  criterion at one scale below
\begin{equation}
\label{CKN}		
\norm{u}_{L^{3}(Q(1))}+\|u\Pi\|_{L^1(Q(1))}+\|\Pi\|_{L^{1,5/4}(Q(1))}\leq \epsilon. \end{equation}
 An alternative  approach  of Caffarelli-Kohn-Nirenberg theorem  based on blow-up argument   was duo to
Lin, Ladyzhenskaya and Seregin \cite{[LS],[Lin]}, where the corresponding $\epsilon$-regularity  criterion at one scale reads
\begin{equation}
\label{Lin}
\norm{u}_{L^{3}(Q(1))}+\|\Pi\|_{L^{3/2}(Q(1))} \leq\epsilon.
\end{equation}
In what follows, a point $z=(x,t)$ in \eqref{NS} is said to be regular  if $u$ belongs to $ L^{\infty}$ at a neighborhood of $z$.
Otherwise, it is called singular. The estimation of the  size of potential singular points set in 3D the Navier-Stokes equations can be found in \cite{[WW],[WZ],[RS1],[RS2],[RS3],[Kukavica],[KP],[GKT],[BW]}.

On the other hand,
 the   integral (Serrin) type conditions based on the  velocity, the  gradient of the velocity or  the pressure leads to the
the full regularity of Leray-Hopf weak solutions of the 3D Navier-Stokes equations.
 Precisely,
a weak solution $u$ is smooth on $(0,T]$ if it satisfies  one of the following three  conditions
  \begin{enumerate}[(1)]
 \item Serrin \cite{[Serrin]}, Struwe \cite{[Struwe]}, Escauriaza,  Seregin and  \v{S}ver\'{a}k \cite{[ESS]}
  \be \label{serrin1}
  u\in  L^{p} (0,T;L^{q}( \mathbb{R}^{3})) ~~~ \text{with}~~~~2/p+3/q=1, ~~q\geq3.
  \ee
 \item  Beirao da Veiga \cite{[Beirao da Veiga]}
  \be \label{serrin2}
  \nabla u\in  L^{p} (0,T;L^{q}( \mathbb{R}^{3})) ~~~ \text{with}~~~~2/p+3/q=2 , ~~q>3/2.
  \ee
 \item  Berselli and  Galdi \cite{[BG]},     Zhou \cite{[Zhou1],[Zhou2]}
  \be \label{serrin3}
 \Pi \in  L^{p} (0,T;L^{q}( \mathbb{R}^{3})) ~~~ \text{with}~~~~2/p+3/q=2 , ~~q>3/2.
  \ee

  \end{enumerate}
 The aforementioned  integral (Serrin) type conditions can be seen as the critical regularity, which is scale  invariant   under the natural scaling of the Navier-Stokes equations \eqref{NS}. The full regularity means that the   set   of $\mathcal{S}$ is empty.  The  natural (supercritical) regularity $u \in L^{q}(\wred{0,\,T};\,L^{p}(\mathbb{R}^{3}))  $ with $\f2q+\f3p=\f32$ in suitable weak solutions means that
 \be\label{clascilresult}
\dim_{H}(\wred{\mathcal{S}})\leq 1~~\text{and}~   \dim_{B}(\wred{\mathcal{S}})\leq5/3,
\ee
which can be found in \cite{[CKN],[RS1]} and  $\dim_{H}(S)$ and $\dim_{B}(S)$ denote the Hausdorff dimension and box dimension of a set $S$, respectively.
 A natural question is weather the suitable weak solutions satisfying  supercritical  regularity $u \in L^{q}(0,\, T;\,L^{p}(\mathbb{R}^{3}))  $ with $1<\f2q+\f3p<\f32$ lower the fractal  dimension in \eqref{clascilresult}. In this direction,
 Gustafson,  Kang and   Tsai  \cite{[GKT]} proved that
   the Hausdorff  dimension of  the potential singular  points set   $\mathcal{S}$  of a Leray-Hopf weak solution belonging in $u\in L^{q}(0,T;L^{p}(\mathbb{R}^{3}))$ for $1\leq\f{2}{q}+\f{  3}{p} $ with $\f2p+\f2q<1$ and $\f3p+\f1q<1$  is at most
   $3-q+\f{2q}{p}, p>q$ or $2-q+\f{3q}{p},  p\leq q$.
Robinson and  Sadowski \cite{[RS3]} showed  that the upper box  dimension of    potential singular points set $\mathcal{S}$ of   a suitable  weak solution belonging in $u\in L^{q}(0,T;L^{p}(\mathbb{R}^{3}))$ for $1\leq\f{2}{q}+\f{  3}{p} \leq \f32$ with $3<p,q<\infty$    is   no greater than
\be\label{fir1b}
\max\{p,q\}(\f{2}{q}+\f{  3}{p}-1). \ee
In addition,   the   Hausdorff dimension of    potential singular  points set  $\mathcal{S}$ of   a suitable  weak solution belonging in $\nabla u\in L^{q}(0,T;L^{p}(\mathbb{R}^{3}))$ for $2\leq\f{2}{q}+\f{  3}{p} \leq \f52$ with $2<p\leq q<\infty$    is   less than or equal to
\be\label{sec2h}
\max\{p,q\}(\f{2}{q}+\f{  3}{p}-2).
\ee
In \cite[Conclusion, Page 9]{[RS3]},
Robinson and  Sadowski
mentioned
some natural questions    from their results :
  \begin{enumerate}[(1)]
 \item
  It would be interesting to relax the
assumption  $q> 3$  in \eqref{fir1b} and obtain the same bound for any $q \geq 2$; \item  similarly
in \eqref{sec2h} one would like to relax the condition $q\geq p$.
 \item
In order to obtain \eqref{fir1b}  in a bounded domain we would require the analogue of Lemma 2 (estimates for the
pressure when $u\in L^{q}(0,T;L^{p}(\Omega))$.
 \item  An order of magnitude harder is to determine whether any of these partial regularity
results can be proved for general weak solutions, and not only suitable weak solutions.
\end{enumerate}
\wred{In this paper, our first result is the following theorem.}
\begin{theorem}\label{the1.1}
Let $u$ be a suitable weak solution belonging in $u\in L^{q}(0,T;L^{p}(\mathbb{R}^{3}))$ for $1\leq\f{2}{q}+\f{  3}{p}\leq\f32$ with $2\leq q<\infty$  and $2<p<\infty$. Then, the upper box dimension of  its potential singular points set   $\mathcal{S}$   is at most  $\max\{p,q\}(\f{2}{q}+\f{  3}{p}-1)$.
		\end{theorem}
\begin{remark}
Theorem \ref{the1.1} answers Robinson and  Sadowski's first question (1).
\end{remark}
As observed in  \cite{[GKT]}, the weak solutions in spaces  $L^{q}(0,T;L^{p}(\mathbb{R}^{3}))$   with $\f2p+\f2q<1$ and $\f3p+\f1q<1$ are suitable weak solutions. Therefore,
towards the Robinson and  Sadowski's fourth question (4), we have
\begin{coro}\label{coro1.2} Let $u$ be a Leray-Hopf weak solution belonging in $u\in L^{q}(0,T;L^{p}(\mathbb{R}^{3}))$ for $1\leq\f{2}{q}+\f{  3}{p}\leq\f32$ with $\f2p+\f2q<1$ and $\f3p+\f1q<1$. Then, the upper box dimension of  its potential singular \wred{points set}  $\mathcal{S}$   is at most  $\max\{p,q\}(\f{2}{q}+\f{  3}{p}-1)$.
\end{coro}
\wred{With a slight modification of} the proof of Theorem \ref{the1.1} and using the $\epsilon$-regularity  criterion at one scale without pressure in \cite{[WWZ]}, we \wred{can obtain}
a parallel result of \eqref{fir1b} in a bounded domain, which \wred{is corresponding} to
   Robinson and  Sadowski's \wred{third} issue.
\begin{theorem}\label{the1.2}
Let $u$ be a suitable weak solution belonging in $u\in L^{q}(0,T;L^{p}(\Omega))$ for $1\leq\f{2}{q}+\f{  3}{p}\leq\f32$ with $\f52<q,p<\infty$. Then, the upper box dimension of  its potential singular  points set   $\mathcal{S}$   is at most  $\max\{p,q\}(\f{2}{q}+\f{  3}{p}-1)$.
		\end{theorem}
Roughly, the following   figures summarize the known upper box dimension of  its potential singular  points set   $\mathcal{S}$  of suitable weak solutions under supercritical    regularity
in the Navier-Stokes equations.

\begin{tikzpicture}[scale=2.3,>= stealth]
   \fill[gray!30](1/9, 1/3)--(1/3,0 )--(1/3,1/4)--(5/18,1/3);
 \filldraw  (1/3,0)--(0,0.5 );
\filldraw  (1/2,0)--(0,3/4);
 \filldraw  (1/3,0)--(1/3,1/4 );
  \filldraw  (1/9,  1/3)--(5/18,1/3 );
 \node[below] at ( 1/3,0) {\tiny{$\frac{1}{3}$}};
 \node[left] at (0,1/2) {\tiny{$\frac{1}{2}$}};
 \node[left] at (0,3/4) {\tiny{$\frac{3}{4}$}};
  \node[left] at (1/9,1/3) {\tiny{$(\f19,\f13) $}};
    \node[right] at (0.22 ,0.37) {\tiny{$(\f{5}{18},\f13) $}};
      \node[right] at (1/3,0.22) {\tiny{$(\f13,\f14) $}};
      \node[below] at (1/2,0) {\tiny{$ \f12 $}};
\draw[->] (0,0)--(1.05,0)node[right]{\tiny{$\frac{1}{p}$}};
\foreach \x in { 1,...,1.05}
     		\draw (\x,0) -- (\x,-.01)
		node[anchor=north] {\tiny{\x}};
\foreach \x in {1,...,1.2}     		\draw (\x,0) -- (\x,-.01)		node[anchor=north] {\tiny{\x}};
\draw[->](0,0)--(0,1.05)node[above]{\tiny{$\frac{1}{q}$}};
 \foreach \y in {0,...,1 }
     		\draw (0,\y) -- (.01,\y)
	 	node[anchor=east] {\tiny{\y}};
 \node[below] at (0.5,-0.2) {\tiny{Fig. 1:  Robinson-Sadowski  }};
 \node[below] at (0.5,-0.35) {\tiny{  results on $\mathbb{R}^{3}$}};
\end{tikzpicture}
\begin{tikzpicture}[scale=2.3,>= stealth]
 \filldraw  (1/3,0)--(0,0.5 );
 \filldraw  (0,1/2)--(1/6,1/2  );
\filldraw  (1/2,0)--(0,3/4);
 \node[below] at ( 1/3,0) {\tiny{$\frac{1}{3}$}};
 \node[left] at (0,1/2) {\tiny{$\frac{1}{2}$}};
 \node[left] at (0,3/4) {\tiny{$\frac{3}{4}$}};
  \node[below] at (1/2,0) {\tiny{$ \f12 $}};
\draw[->] (0,0)--(1.05,0)node[right]{\tiny{$\frac{1}{p}$}};
\foreach \x in {1,...,1.05}
     		\draw (\x,0) -- (\x,-.01)
		node[anchor=north] {\tiny{\x}};
\foreach \x in {1,...,1.2}     		\draw (\x,0) -- (\x,-.01)		node[anchor=north] {\tiny{\x}};
\draw[->](0,0)--(0,1.05)node[above]{\tiny{$\frac{1}{q}$}};
 \foreach \y in {0,...,1.05}
     		\draw (0,\y) -- (.01,\y)
	 	node[anchor=east] {\tiny{\y}};
\node[below] at (0.5,-0.2) {\tiny{Fig. 2:  Theorem 1.1 }};
 \node[below] at (0.7,-0.35) {\tiny{    on  $\mathbb{R}^{3}$}};
  \fill[gray!30](1/3, 0)--(1/2,0 )--( 1/6,1/2)--(0,1/2);
\end{tikzpicture}\begin{tikzpicture}[scale=2.3,>= stealth]
 \filldraw  (1/3,0)--(0,0.5 );
\filldraw  (1/2,0)--(0,3/4);
\filldraw  (0, 1/2)--(1/4,1/4);
\filldraw  (1/3,0)--(1/4,1/4);
 \node[right] at (1/4,1/4) {\tiny{$(\frac{1}{4},\frac{1}{4})$}};;
 \node[below] at ( 1/3,0) {\tiny{$\frac{1}{3}$}};
 \node[left] at (0,1/2) {\tiny{$\frac{1}{2}$}};
 \node[left] at (0,3/4) {\tiny{$\frac{3}{4}$}};
  \node[below] at (1/2,0) {\tiny{$ \f12 $}};
\draw[->] (0,0)--(1.05,0)node[right]{\tiny{$\frac{1}{p}$}};
\foreach \x in { 1,...,1.05}
     		\draw (\x,0) -- (\x,-.01)
		node[anchor=north] {\tiny{\x}};
\foreach \x in {1,...,1.05}     		\draw (\x,0) -- (\x,-.01)		node[anchor=north] {\tiny{\x}};
\draw[->](0,0)--(0,1.05)node[above]{\tiny{$\frac{1}{q}$}};
 \foreach \y in {0,...,1.05}
     		\draw (0,\y) -- (.01,\y)
	 	node[anchor=east] {\tiny{\y}};
\node[below] at (0.5,-0.2) {\tiny{Fig. 3: Corollary 1.2 }};
 \node[below] at (0.5,-0.35) {\tiny{    weak solutions}};
  \fill[gray!30](1/3, 0)--(1/4,1/4 )--(0,1/2);
\end{tikzpicture}
\begin{tikzpicture}[scale=2.3,>= stealth]
 \filldraw  (1/3,0)--(0,0.5 );
\filldraw  (1/2,0)--(0,3/4);
\filldraw  (2/5,0)--(2/5,3/20);
\filldraw  (7/30,2/5  )--( 1/15, 2/5);
 \node[below] at (1/2,0) {\tiny{$ \f12 $}};
 \node[below] at ( 1/3,0) {\tiny{$\frac{1}{3}$}};
 \node[left] at (0,1/2) {\tiny{$\frac{1}{2}$}};
 \node[left] at (0,3/4) {\tiny{$\frac{3}{4}$}};
    \node[below] at (2/5,0) {\tiny{$ \f25  $}};
        \node[right] at (2/5, 3/20 ) {\tiny{$(\f25,\f{3}{20}) $}};
          \node[right] at (7/30, 2/5 ) {\tiny{$(\f{7}{30}, 2/5) $}};
           \node[left] at (0.1, 2/5 ) {\tiny{$(\f{1}{15},\f25) $}};
\draw[->] (0,0)--(1.05,0)node[right]{\tiny{$\frac{1}{p}$}};
\foreach \x in { 1,...,1.05}
     		\draw (\x,0) -- (\x,-.01)
		node[anchor=north] {\tiny{\x}};
\foreach \x in {1,...,1.05}     		\draw (\x,0) -- (\x,-.01)		node[anchor=north] {\tiny{\x}};
\draw[->](0,0)--(0,1.05)node[above]{\tiny{$\frac{1}{q}$}};
 \foreach \y in {0,...,1.2}
     		\draw (0,\y) -- (.01,\y)
	 	node[anchor=east] {\tiny{\y}};
\fill[gray!30](1/15, 2/5)--(1/3,0 )--(2/5,0 )--(2/5,3/20)--(7/30,2/5);
 \node[below] at (0.5,-0.2) {\tiny{ Fig. 4:  Theorem 1.3 }};
 \node[below] at (0.5,-0.35) {\tiny{    on bounded domain}};
\end{tikzpicture}
Next, we study the   Robinson and  Sadowski's second issue  involving the gradient of  the velocity with additional regularity.
It seems that this problem is more complicated. Very recently, in the other direction,  Baker and Wang \cite{[BW]} estimate the  Hausdorff dimension of  the singular set  for the Navier-Stokes equations with supercritical assumptions on the pressure.  There are two new   ingredients in their proof. The first one is is the high  regularity of the
solutions with certain \wred{supercritical} assumptions on pressure in  the Navier-Stokes equations. The second one is the $\epsilon$-regularity  criterion in terms of  quantity $|\nabla u|^{2}|v|^{q-2}$ with $2<q<3$, which usually arises in the $L^{p}$ type energy estimates of the Navier-Stokes equations.  In the spirit of \cite{[BW]}, we consider the $\epsilon$-regularity  criterion via  quantity $  \Lambda^{s+1}u$ with $s>0$, which usually appears in the $\dot{H}^{s+1}$   type energy estimates of the Navier-Stokes equations. One naturally invokes  the     Caffarelli-Silvestre extension used in \cite{[RWW],[TY],[CDM]}  to overcome non-local derivatives. However, since $s>0$, one requires   higher order Caffarelli-Silvestre (Yang) extensions \cite{[Yang]}. To this end, we observe that
  that  the following identity due to \cite{[CDM]}, for $\alpha=s+1>1$,
\begin{equation*}
c_\alpha\int_{\mathbb R^4_+} y^{3-2\alpha} |\nabla^{\ast} (\nabla u)^{\ast}|^2 (x,y,t)\, dx\, dy
= \int_{\mathbb R^3} \left|(-\Delta)^{\frac{\alpha-1}{2}} \nabla u\right|^2 (x,t)\, dx=  \int_{\mathbb R^3} |(-\Delta)^\frac{\alpha}{2} u|^2 (x,t)\, dx,
\end{equation*}
that is,
 \be\label{CSK1}
\|u\|^{2} _{\dot{H}^{s+1}}= c_s\int_{\mathbb R^4_+} y^{1-2s} |\nabla^{\ast} (\nabla u)^{\ast}|^2 (x,y,t)\, dx\, dy,\ee
which helps us to reduce the proof of Theorem \ref{the1.4} to show  Theorem \ref{the1.5} just by Caffarelli-Silvestre extension rather than higher order (Yang) extension.
 Theorem \ref{the1.4} can be viewed  as the interpolation between the Caffarelli-Kohn-Nirenberg theorem and  Kozono-Taniuchi regular class
 $L^{2}(0,T; BMO)$, which is   of independent interest.
\begin{theorem}\label{the1.4}
Let $u$ be a suitable weak solution belonging in $  u\in L^{2}(0,T;\dot{H}^{s+1}(\mathbb{R}^{3}))$ for $0\leq s\leq\f12$. Then,  $\mathcal{H}^{ 1-2 s}(\mathcal{S})=0$.
		\end{theorem}
 \begin{theorem}
\label{the1.5}
Suppose that $u$ is a suitable weak solution to \wred{(\ref{NS})}. Then there exists an absolute positive constant $\varepsilon_{01}$ such that $(0,0)$ is a regular point if
$$
\frac{1}{\mu^{1-2s}}\iint_{Q^{\ast}(\mu)}y^{1-2s}|\nabla^{\ast} (\nabla u)^{\ast}|^2dxdydt  \leq\varepsilon_{01}.
$$
\end{theorem}
As an application of  Theorem \ref{the1.4} and the energy estimate of the Navier-Stokes equtions, we can partially answer the Robinson and  Sadowski's \wred{second} question.
\begin{coro}\label{coro1.6}
Let $u$ be a suitable weak solution belonging in $\nabla u\in L^{q}(0,T;L^{p}(\mathbb{R}^{3}))$ for $2\leq\f{2}{q}+\f{  3}{p}\leq\f52$\wred{.}

(1) \wred{If} $\f52-\f3p-\f{5}{2q}\geq0,1<p<\f{54+12\sqrt{14}}{25} ,\wred{1<q\leq2}$\wred{, then}  $\mathcal{H}^{ \f{2(\f{2}{q}+\f{  3}{p}-2)}{1-\f1q}}(\mathcal{S})=0$.

(2)  If  $2-\f3p-\f1q\geq0,\f32<p<\f{12}{7},q\geq4$, then  $\mathcal{H}^{q(\f{2}{q}+\f{  3}{p}-2)}(\mathcal{S})=0$.
		\end{coro}
At present, the Hausdorff dimension of     suitable weak solutions with the gradient of the velocity under  supercritical    regularity are   summarized  in the following figures.

\begin{tikzpicture}[scale=3,>= stealth]
 \fill[gray!30](1/3, 1/2)--(5/12,1/2 )--(0.2578,0.69064)--(0.2578,0.6133);
   \node[below] at (1/2,0) {\tiny{$ \f12 $}};
      \node[left] at (0,1/2) {\tiny{$ \f12 $}};
\draw[->] (0,0)--(1.1,0)node[below]{\tiny{$\f1p$}};
\node[below] at (0,-0.05) {\tiny{0}};
\foreach \x in { 1, ...,1.1}
     		\draw (\x,0) -- (\x,-.01)
		node[anchor=north] {\tiny{\x}};
\draw[->](0,0)--(0,1.4)node[right]{\tiny{$\f1q$}};
\foreach \y in { 1, ...,1.4}
     		\draw (0,\y) -- (.01,\y)
		node[anchor=east] {\tiny{\y}};
\draw(0,1)--(0.6666,0 );
\draw (0,1.25)--(0.83333,0);
\draw (0,1)--(0.83333,0);
\draw (1/3, 1/2)--(5/12,1/2 );
 \node[below] at ( 2/3,0) {\tiny{$\frac{2}{3}$}};
 \node[below] at ( 5/6,0) {\tiny{$\frac{5}{6}$}};
  \node[left] at ( 0,5/4) {\tiny{$\frac{5}{4}$}};
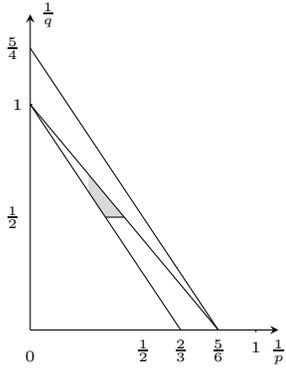
  \node[below] at (0.5,-0.2) {\tiny{Fig. 5: First part of \wred{Corollary 1.6} }};
  \end{tikzpicture}
\begin{tikzpicture}[scale=3,>= stealth]
 \fill[gray!30](2/3,0)--(7/12,  1/8)--(7/12,1/4);
\draw[->] (0,0)--(1.1,0)node[below]{\tiny{$\f{1}{p}$}};
\node[below] at (0,-0.05) {\tiny{0}};
\foreach \x in { 1, ...,1.1}
     		\draw (\x,0) -- (\x,-.01)
		node[anchor=north] {\tiny{\x}};
\draw[->](0,0)--(0,1.4)node[right]{\tiny{$\f{1}{q}$}};
\foreach \y in { 1, ...,1.4}
     		\draw (0,\y) -- (.01,\y)
		node[anchor=east] {\tiny{\y}};
\draw (0,1)--(0.6666,0 );
\draw (0,1.25)--(0.83333,0);
\draw (0.5,0.5)--(2/3,0);
\draw ( 0,0.25)--(7/12,0.25);
 \node[below] at ( 2/3,0) {\tiny{$\frac{2}{3}$}};
\node[below] at ( 5/6,0) {\tiny{$\frac{5}{6}$}};
  \node[left] at ( 0,5/4) {\tiny{$\frac{5}{4}$}};
  \node[left] at ( 0,1/4) {\tiny{$\frac{1}{4}$}};
    \node[below] at (1/2,0) {\tiny{$ \f12 $}};
      \node[left] at (0,1/2) {\tiny{$ \f12 $}};
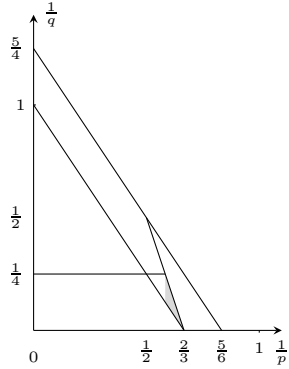
   \node[below] at (0.5,-0.15) {\tiny{Fig. 6: Second part of \wred{Corollary 1.6}}};
 \end{tikzpicture}
\begin{tikzpicture}[scale=3,>= stealth]
 \fill[gray!100](2/3,0)--(7/12,  1/8)--(7/12,1/4);
\draw[->] (0,0)--(1.1,0)node[below]{\tiny{$\f1p$}};
  \node[below] at (1/2,0) {\tiny{$ \f12 $}};
      \node[left] at (0,1/2) {\tiny{$ \f12 $}};
\foreach \x in { 1,...,1.1}
     		\draw (\x,0) -- (\x,-.01)
		node[anchor=north] {\tiny{\x}};
\foreach \x in {1,...,1.2}     		\draw (\x,0) -- (\x,-.01)		node[anchor=north] {\tiny{\x}};
\draw[->](0,0)--(0,1.4)node[right]{\tiny{$\f1q$}};
 \foreach \y in { 1,...,1.4}
     		\draw (0,\y) -- (.01,\y)
	 	node[anchor=east] {\tiny{\y}};
\draw  (0,1)--(0.6666,0 );
\draw (0,1.25)--(0.83333,0);
\draw (0.5,0.5)--(2/3,0);
\draw ( 0,0.25)--(7/12,0.25);
 \node[below] at ( 2/3,0) {\tiny{$\frac{2}{3}$}};
   \fill[gray!100](1/3, 1/2)--(5/12,1/2 )--(0.2578,0.69064)--(0.2578,0.6133);
\fill[gray!100](1/2, 1/2)--(2/5,2/5 )--(1/2,1/4);
 \node[below] at (0.5,-0.2) {\tiny{Fig. 7:   Known Hausdorff dimension of}};
\node[below] at (0.5,-0.3) {\tiny{the gradient of the velocity }};
 \draw (0.33,0.57)--(2/3,0.8);
 \node[right] at (0.66,0.9) {\tiny{$\mathcal{H}^{ \f{2(\f{2}{q}+\f{  3}{p}-2)}{1-\f1q}}(\mathcal{S})=0$}};
 \draw (0.46,0.4)--(0.8,0.34 );
 \draw (0.6,0.14 )--(0.8,0.34 );
 \node[right] at (0.8,0.34 ) {\tiny{$\mathcal{H}^{q(\f{2}{q}+\f{  3}{p}-2)}(\mathcal{S})=0$}};
  \end{tikzpicture}

The remainder of this paper is organized as follows. In \wred{Section}
\ref{sec2}, we will begin with the notations and the definition of fractal dimension including the Box dimension  and  Hausdorff dimension. Then we recall the  Caffarelli and Silvestre's generalized extension for the fractional
Laplacian operator and  $\epsilon$-regularity  criterion at one scale.
Section \ref{sec3} is \wred{devoted} to the proof to Theorem 1 \wred{concerning} Box dimension. Partial regularity results involving  Hausdorff dimension is proved in \wred{Section} \ref{sec4}.
 \section{ Preliminaries}\label{sec2}
\setcounter{section}{2}\setcounter{equation}{0}
 First, we introduce some notations used in this paper.
Throughout this paper, we denote
$$\ba
&B(x,\mu)=\{y\in \mathbb{R}^{3}||x-y|\leq \mu\},~~ &B&(\mu):= B(0,\mu),\\
&Q(x,t,\mu)=B(x,\,\mu)\times(t-\mu^{2 }, t),~~ &Q&(\mu):= Q(0,0,\mu)
,\\
&\wred{B^{\ast}(x,\mu)=B(x,\mu)\times(0,\mu),}~~ &\wred{B^{\ast}}&\wred{(\mu):= B^{\ast}(0,\mu),}\\
&Q^{\ast}(x,t,\mu)=B(x,\mu)\times(0,\mu)\times(t-\mu^{2 },t),~~ &Q^{\ast}&(\mu):= Q^{\ast}(0,0,\mu).
\ea$$
 For $p\in [1,\,\infty]$, the notation $L^{p}(0,\,T;X)$ stands for the set of measurable functions on the interval $(0,\,T)$ with values in $X$ and $\|f(\cdot,t)\|_{X}$ belonging to $L^{p}(0,\,T)$.
  For simplicity,   we write $\|f\| _{L^{p,q}(Q(\mu))}:=\|f\| _{L^{p}(\wred{-\mu^{2}},0;L^{q}(B(\mu)))} $ and
 $\|f\| _{L^{p}(Q(\mu))}:=\|f\| _{L^{p}L^{p}(Q(\mu))}$. We shall denote by $\langle f,\,g\rangle$ the $L^{2}$ inner product of $f$ and $g$.
 The classical Sobolev norm $\|\cdot\|_{H^{s}}$  is defined as   $\|f\|^{2} _{{H}^{s}}= \int_{\mathbb{R}^{n}} (1+|\xi|)^{2s}|\hat{f}(\xi)|^{2}d\xi$, $s\in \mathbb{R}$.
  We denote by  $ \dot{H}^{s}$ homogenous Sobolev spaces with the norm $\|f\|^{2} _{\dot{H}^{s}}= \int_{\mathbb{R}^{n}} |\xi|^{2s}|\hat{f}(\xi)|^{2}d\xi$. Denote
  the average of $f$ on the ball $B(\mu)$ by
$\overline{f}_{\mu}$. $\Gamma$  denotes the standard normalized fundamental solution of Laplace equation in $\mathbb{R}^{3}$.
We denote by  $\D$ the divergence operator in $\mathbb{R}^{4}_{+}$.
  $|S|$ represents the Lebesgue measure of the set $S$. We will use the summation convention on repeated indices.
 $C$ is an absolute constant which may be different from line to line unless otherwise stated in this paper.
\begin{definition}\label{defibox}
The (upper) box-counting dimension of a set $X$ is usually defined as
$$\wred{\mathrm{dim}_B(X)}=\limsup_{\epsilon\rightarrow0}\f{\log N(X,\,\epsilon)}{-\log\epsilon},$$
where $N(X,\,\epsilon)$ is the minimum number of balls of radius $\epsilon$ required to cover $X$.
\end{definition}
\wred{Let $\beta>0$, $\delta>0$} and $\Omega\times I$ can be covered by the \wred{union of series of} parabolic balls $Q(r)$ with radius $r_{j}$ less than $\delta$ for $j\in \mathbb{N}$. Define
$$
\mathcal{P}^{\beta}_{\delta}= \inf\B\{\Sigma r_{j}^{\beta}| \Omega\times I\subseteq \cup Q(r_j), r_{j}<\delta,j \in \mathbb{N}\B\}
$$
and $\mathcal{P}^{\beta}=\lim_{\delta\rightarrow0}\mathcal{P}^{\beta}_{\delta}$. If there is $\beta_{0}$ such that \wred{$\mathcal{P}^{\beta}=\infty$} if $\beta<\beta_{0}$ and \wred{$\mathcal{P}^{\beta}=0$} if $\beta>\beta_{0}$\wred{, then} $\beta_{0}$ is called as the parabolic Hausdorff dimension and \wred{$\mathcal{P}^{\beta}$} is the parabolic Hausdorff measure. The details of fractal dimension  can be found in \cite{[Falconer]}.

Next, we focus on      Caffarelli and Silvestre's generalized extension for the fractional Laplacian operator $(-\Delta)^{\alpha}$ with $0<\alpha<1$ in \cite{[CS]}. The  fractional power of  Laplacian in $\mathbb{R}^{3}$ can be interpreted as
$$
(-\Delta)^{\alpha}u=-C_{\alpha}\lim\limits
_{y\rightarrow0_{+}}y^{1-2\alpha}\partial_{y}u^{\ast},
$$
where $u^{\ast}$ satisfies
\be\label{csex}\left\{\ba
  &\D(y^{1-2\alpha}\nabla^{\ast} u^{\ast})=0 ~~\text{in}~~ \mathbb{R}^{4}_{+},\\  &u^{\ast}|_{y=0}=u,~x\in\mathbb{R}^{3}. \ea\right.
\ee
As a by-product of the above equation, for any $v |_{y=0}=u$, it holds
\be\label{mini}
\int_{\wred{\mathbb{R}_{+}^{4}}} y^{1-2s} | \nabla^{\ast} u^{\ast} |^{2}dxdy\leq \int_{\wred{\mathbb{R}_{+}^{4}}} y^{1-2s} | \nabla^{\ast}v  |^{2}dxdy.
\ee

Moreover, from  Section 3.2 in \cite{[CS]}, the definition of the $\dot{H}^{\alpha}$ norm can be written as
\be\label{eqnorm}\|u\|^{2} _{\dot{H}^{\alpha}}= \int_{\mathbb{R}^{3}} |\xi|^{2\alpha}|\hat{u}(\xi)|^{2}d\xi=\int_{\mathbb{R}_{+}^{4}} y^{1-2\alpha} | \nabla^{\ast} u^{\ast} |^{2}dxdy.
\ee
We recall the following observation due to \cite{[CDM]}, for $\alpha>1$,
\begin{equation*}
c_\alpha\int_{\mathbb R^4_+} y^{3-2\alpha} |\nabla^{\ast} (\nabla u)^{\ast}|^2 (x,y,t)\, dx\, dy
= \int_{\mathbb R^3} \left|(-\Delta)^{\frac{\alpha-1}{2}} \nabla u\right|^2 (x,t)\, dx=  \int_{\mathbb R^3} |(-\Delta)^\frac{\alpha}{2} u|^2 (x,t)\, dx\,.
\end{equation*}
Hence,

\be\label{CSK}
\|u\|^{2} _{\dot{H}^{s+1}}= c_s\int_{\mathbb R^4_+} y^{1-2s} |\nabla^{\ast} (\nabla u)^{\ast}|^2 (x,y,t)\, dx\, dy.\ee
Base on the natural scaling of the Navier-Stokes equations, we set  the following two dimensionless quantities
\begin{align}
&E^{\ast}_{\ast}(\nabla^{\ast} (\nabla u)^{\ast},\mu)=\frac{1}{\mu^{1-2s}}\iint_{Q^{\ast}(\mu)}y^{1-2s}|\nabla^{\ast} (\nabla u)^{\ast}|^2dxdydt,& &\wred{E_{\ast}(\nabla u,\mu)=\frac{1}{\mu}\iint_{Q(\mu)}|\nabla u|^2dxdt}.\nonumber
\end{align}
 To make our paper more self-contained and more readable,
we   outline the   proof of  Poincar\'e inequality concerning Caffarelli and Silvestre's generalized extension.
\begin{lemma}
Let  $u$ and $u^{\ast}$ be defined in  \eqref{csex}. There exist a constant $C$ such that
\begin{align}
&\|u-\overline{u}_{\mu}\|_{L^{\f{6}{3-2s}}(B(\mu/2))}\leq C\B(\int_{B^{\ast}(\mu)}y^{1-2s}|\nabla^{\ast} u^{\ast}|^{2}dxdy\Big)^{1/2},\label{akeyinq1}\\
&\|u-\overline{u}_{\mu}\|_{L^{2}(B(\mu/2))}\leq C\mu^{s}\B(\int_{B^{\ast}(\mu)}y^{1-2s}|\nabla^{\ast} u^{\ast}|^{2}dxdy\Big)^{1/2}.\label{keyinq1}
\end{align}\end{lemma}

\begin{proof}
Consider the usual cut-off \wred{functions}
$$\eta_{1}(x)=\left\{\ba
&1,\,\wred{x\in} B(\hbar\mu ),~0<\hbar<1,\\
&0,\,\wred{x\in} B^{c}(\mu),
\ea\right.
$$
and
$$\eta_{2}(y)=\left\{\ba
&1,\, 0\leq y\leq \hbar\mu,\\
&0,\, y>\mu,
\ea\right.
$$
 satisfying
$$
0\leq \eta_{1},\,\eta_{2}\leq1,
~~~\text{and}~~~\mu |\partial_{x}\eta_{1}(x)|
+\mu|\partial_{y}\eta_{2}(y)|\leq C.
$$
By the \wred{Young} inequality and \eqref{eqnorm} and \eqref{mini},
$$\ba
\|u\eta_{1}\|^{2}_{\dot{H}^{s}}&=
\int_{\wred{\mathbb{R}_{+}^{4}}} y^{1-2s} | \nabla^{\ast} (u\eta_{1})^{\ast} |^{2}dxdy\\
&\leq C\int_{\wred{\mathbb{R}_{+}^{4}}} y^{1-2s} | \nabla^{\ast}   (u^{\ast}\eta_{2}\eta_{1})|^{2}dxdy\\
&\leq C\mu^{-2}\int_{B^{\ast}(\mu)}y^{1-2s}|u^{\ast}|^{2}+
C\int_{B^{\ast}(\mu)}y^{1-2s}|\nabla^{\ast}u^{\ast}|^{2}.
\ea$$
Thanks to the classical weight Poincar\'e inequality, we infer that
\be\label{app2}
\int_{B^{\ast}(\mu)}y^{1-2s}|u^{\ast}- \overline{u^{\ast}}_{B^{\ast}(\mu)}|^{2}\leq C\mu^{2}\int_{B^{\ast}(\mu)}y^{1-2s}|\nabla^{\ast}u^{\ast}|^{2},
\ee
where $\overline{u^{\ast}}_{B^{\ast}(\mu)}=\f{1}{|B^{\ast}(\mu)|} \int_{B^{\ast}(\mu)}y^{1-2s} u^{\ast}dxdy$ and
$|B^{\ast}(\mu)|=\int_{B^{\ast}(\mu)} y^{1-2s} dydx.$
The above inequalities imply
\be\B(\int_{B(\hbar\mu)}|u-\overline{u^{\ast}}_{B^{\ast}(\mu)}|
^{\wred{\f{6}{3-2s}}}\B)^{\wred{\f{3-2s}{3}}}
\leq C \int_{B^{\ast}(\mu)}y^{1-2s}|\nabla^{\ast}u^{\ast}|^{2}.
\label{app1}\ee
It follows from  $u^{\ast}  = u(x)+\int^{y}_{0}\partial_{z}u^{\ast}dz$ and the H\"older inequality that
\begin{align}
\wred{\big|}\overline{u^{\ast}}_{B^{\ast}(\mu)}-\overline{u}_{\mu}\wred{\big|}&=
\f{1}{|B^{\ast}(\mu)|} \wred{\big|}\int_{B^{\ast}(\mu)}y^{1-2s}
\int^{y}_{0}\partial_{z}u^{\ast}dz\wred{\big|} \nonumber\\
&\leq C\f{1}{|B^{\ast}(\mu)|}\int_{B (\mu)}\int_{0}^{\mu}\wwred{y^{1-2s}\B(
\int^{y}_{0}z^{1-2s}|\partial_{z}u^{\ast}|^{2}dz\B)^{1/2}
\B(\int^{y}_{0}z^{-(1-2s)}dz\B)^{1/2}dydx} \nonumber\\
&\leq C\mu^{\wred{s-\f{3}{2}}}\B(
\int_{B^{\ast}(\mu)} z^{1-2s}|\nabla^{\ast } u^{\ast}|^{2}dxdz\B)^{1/2}.
\label{app3}
\end{align}
 Combining
\eqref{app1} with the latter inequality, we deduct that
$$\ba
\B(\int_{B(\hbar\mu)}|u-\overline{u}_{\mu}|^{\wred{\f{6}{3-2s}}}\B)^{\wred{\f{3-2s}{6}}}
\leq&
\B(\int_{B(\hbar\mu)}|u-\overline{u^{\ast}}_{B^{\ast}(\mu)}|^{\wred{\f{6}{3-2s}}}\B)^{\wred{\f{3-2s}{6}}}
\\&+ \B(\int_{B(\hbar\mu)}|\overline{u^{\ast}}_{B^{\ast}(\mu)}-\overline{u}_{\mu}|^{\wred{\f{6}{3-2s}}}\B)^{\wred{\f{3-2s}{6}}}
\\
\leq& C \wred{\bigg(\int_{B^{\ast}(\mu)}y^{1-2s}|\nabla^{\ast}u^{\ast}|^{2}\bigg)^{\frac{1}{2}}},
\ea$$
which means \wred{\eqref{akeyinq1} and \eqref{keyinq1}}.
\end{proof}
\begin{prop}(\cite{[HWZ]})\label{hwz}
Let  the pair $(u,  \Pi)$ be a suitable weak solution to the 3D Navier-Stokes system \eqref{NS} in $Q(1)$.
       There exists an absolute positive constant $\varepsilon$ depending only on $ p$ and $q$ \wred{such that} if the pair $(u,\Pi)$ satisfies	\be\label{optical}\|u\|_{\wred{L^{q,p}(Q(1))}}+\|\Pi\|_{L^{1}(Q(1))}<\varepsilon,\ee
		for $1\leq 2/q+3/p <2, 1\leq p,\,q\leq\infty$, then $u\in L^{\infty}(Q(1/2))$.
\end{prop}
\begin{prop}(\cite{[WWZ]})\label{wwz1}
Let  the pair $(u,  \Pi)$ be a suitable weak solution to the 3D Navier-Stokes system \eqref{NS} in $Q(1)$.  For any  $\delta>0$, there exists an absolute positive constant $\varepsilon$
		such that if   $u$ satisfies
 \be\label{wwz}
 \iint_{Q(1)}|u|^{\f{5}{2}+\delta}dxdt\leq \varepsilon,\ee
		then, $u\in L^{\infty}(Q(1/16)).$
\end{prop}
\begin{lemma}(Kato-Ponce Commutator and Product Estimates \cite{[KP1]})
Let $\alpha>0$, $p\in(1,\infty)$ and $p_{i}\in(1,\infty)$, $i=1,2,3,4.$  Then there
exists a positive constant $C$ such that
\be\label{KPC}
\|\Lambda^{\alpha}(fg)-f\Lambda^{\alpha}g\|_{L^{p}}\leq C(\|\nabla f\|_{L^{p_{1}}}\|\Lambda^{\alpha-1}g\|_{L^{p_{2}}}+
\|  \Lambda^{\alpha}f\|_{L^{p_{3}}}\|g\|_{L^{p_{4}}})
\ee
and
\be\label{KPI}
\|\Lambda^{\alpha}(fg)\|_{L^{p}}\leq C
(\|\Lambda^{\alpha}f\|_{L^{p_{1}}}\|g\|_{L^{p_{2}}}+
\|f\|_{L^{p_{3}}}\|\Lambda^{\alpha}g\|_{L^{p_{4}}}),
\ee
where $\f1p=\frac{1}{p_{1}}+\frac{1}{p_{2}}=\frac{1}{p_{3}}+\frac{1}{p_{4}}$.
\end{lemma}
\section{Box dimension of possible singular points set of suitable weak solutions}  \label{sec3}
\setcounter{section}{3}\setcounter{equation}{0}
This section contains the proof of Theorem \ref{the1.1}, Corollary \ref{coro1.2} and Theorem \ref{the1.2}. The key point is an application of the $\epsilon$-regularity  criterion  \eqref{optical} and \eqref{wwz} at one scale.

\begin{proof}[Proof of Theorem \ref{the1.1}]
We present the   proof by contradiction. We suppose that $ \dim_{B}(\wred{\mathcal{S}})> \max\{p,q\}(\f{2}{q}+\f{  3}{p}-1). $ We pick a constant $\alpha$ such that  $\alpha_{0}=\max\{p,q\}(\f{2}{q}+\f{  3}{p}-1)<\alpha<\dim_{B}(\wred{\mathcal{S}})$. Therefore, using the definition of the box dimension, we know that there exists a sequence  $\delta_{j}\rightarrow0$ such that $N(\mathcal{S},\delta_{j})>\delta_{j}^{-\alpha}.  $We assume that
$(x_{i},t_{i})_{i=1}^{N(\mathcal{S},\delta_{j})}$ be a collection of $\delta_{j}$- separated points in $\mathcal{S}$. By the regularity criterion
 in \wred{Proposition \ref{hwz}}, for any $(x_{i},t_{i})\in \mathcal{S}$, we get
$$\ba
 \int_{t_{i}-\wred{\delta_{j}^{2}}}^{t_{i}}
\B[\B(\int_{B_{i} ( \delta_{j})}
|  u |^{p}  dx\B)^{\f{q}{p}} \wred{+}
\B(\int_{B_{i} ( \delta_{j})}
|  \Pi |^{p/2}  dx\B)^{\f{q}{p}}\B]dt
> \delta_{j}^{(-p+3)\f{q}{p}+2}\varepsilon_{1},  \ea$$
where $B_{i}(\mu):=B(x_{i}, \mu)$. \wred{Thus we have}
\be\label{bh3.1}\ba
 \sum^{N(\mathcal{S},\,\wred{\delta_{j}})}_{i=1 } \int_{t_{i}-\wred{\delta_{j}^{2}}}^{t_{i}}
\B[\B(\int_{B_{i} ( \delta_{j})}
|  u |^{p}  dx\B)^{\f{q}{p}} \wred{+}
\B(\int_{B_{i} ( \delta_{j})}
|  \Pi |^{p/2}  dx\B)^{\f{q}{p}}\B]dt
>N(\mathcal{S},\,\wred{\delta_{j}}) \delta_{j}^{(-p+3)\f{q}{p}+2}\varepsilon_{1}\wred{.}
 \ea\ee
 The pressure equation help us to obtain, for $p>2$, \wred{$q\geq 2$,}
\be\label{pres}\|\Pi\|^{p/2}_{L^{q/2}(0,T;L^{p/2}(\mathbb{R}^{3}))}\leq C\|u\|^{p}_{L^{q}(0,T;L^{p}(\mathbb{R}^{3}))}. \ee
For the case $\f{q}{p}>1$, we know that $\alpha_{0}=q(\f{2}{q}+\f{  3}{p}-1)$.

Now, we can apply the inequality $ \sum^{N(\mathcal{S},\,\wred{\delta_{j}})}_{i=1 } (a_{i})^{\f{q}{p}}\leq (\sum^{N(\mathcal{S},\,\wred{\delta_{j}})}_{i=1 }a_{i})^{\f{q}{p}}$ to control the left hand side of  \eqref{bh3.1} by $\|u\|^{\wred{q}}_{L^{q}(0,T;L^{p}(\mathbb{R}^{3}))} +\|\Pi\|^{\wred{q/2}}_{L^{q/2}(0,T;L^{p/2}(\mathbb{R}^{3}))}$. This together with \eqref{pres} implies that
\be
C \geq N(\mathcal{S},\,\wred{\delta_{j}}) \delta_{j}^{(-p+3)\f{q}{p}+2}\varepsilon_{1}\geq\delta_{j}^{(-p+3)\f{q}{p}+2-\alpha}\varepsilon_{1}\wred{.}
\ee
We immediately get a contradiction as $j\rightarrow\infty.$

For the rest case $\f{q}{p}\leq1$, we invoke the inequality  $ \sum^{N(\mathcal{S},\,\wred{\delta_{j}})}_{i=1 } (a_{i})^{\f{q}{p}}\leq N^{(1-\f{q}{p})}(\mathcal{S},\,\wred{\delta_{j}})(\sum^{N(\mathcal{S},\,\wred{\delta_{j}})}_{i=1 }a_{i})^{\f{q}{p}}$ in the proof. \wred{
With a slight modification of } the above proof, we see that $C \geq  N^{\f{q}{p}}(\mathcal{S},\,\wred{\delta_{j}}) \delta_{j}^{(-p+3)\f{q}{p}+2}\varepsilon_{1}$.
This means that
\be
C \geq N(\mathcal{S},\,\wred{\delta_{j}}) \delta_{j} ^{\f{p}{q}[(-p+3)\f{q}{p}+2]} \varepsilon_{1}\geq\delta_{j}^{\f{p}{q}[(-p+3)\f{q}{p}+2]-\alpha}\varepsilon_{1}\wred{.}
\ee
This lead a contradiction as $j\rightarrow\infty$. The proof of this theorem is achieved.
\end{proof}

\begin{proof}[Proof of Corollary \ref{coro1.2}]
It follows from
 $1\leq\f{2}{q}+\f{  3}{p}\leq\f32$ with $\f2p+\f2q<1$ and $\f3p+\f1q<1$ that $u\in L^{4}(0,T;L^{4}(\mathbb{R}^{n}))$. Thanks to the work \cite{[Taniuchi]}, we observe \wred{that} $u$ is  a suitable weak solution. Following the path of the above proof, we complete the proof.
\end{proof}

\begin{proof}[Proof of Theorem \ref{the1.2}]
As the same manner of proof of Theorem \ref{the1.1} and
  replacing the appliction of  the regularity criterion \eqref{optical}
  by  \eqref{wwz}, the proof of this theorem is completed.
\end{proof}

\section{ Hausdorff dimension of possible singular points set of suitable weak solutions}
 \label{sec4}
\setcounter{section}{4}\setcounter{equation}{0}
First, we prove  Theorem  \ref{the1.5}. As an application of this theorem, we can achieve    the proof of  Corollary \ref{coro1.6}.       To this end, we prove the following lemma, which roughly indicates that the smallness of $\nabla^{\ast} (\nabla u)^{\ast}$  yields the
the smallness of $\nabla u$.
\begin{lemma}\label{lem3.3}For $0<\mu\leq\f{1}{2}\rho$,
there is an absolute constant $C$  independent of  $\mu$ and $\rho$,~ such that
$$
E_{\ast}(\nabla u;\mu)\leq(\f{\rho}{\mu})E^{\ast}_{\ast}(\nabla^{\ast} (\nabla u)^{\ast};\rho)+(\f{\mu}{\rho})^{2}E_{\ast}(\nabla u;\rho).
$$
\end{lemma}
\begin{proof}
With the help of the triangle inequality, the H\"older inequality and \eqref{keyinq1}, we see that
\be\label{lem2.31} \begin{aligned}
\int_{B(\mu)}|u|^{2}dx &\leq C\int_{B(\mu)}|u-\bar{u}_{{\rho}}|^{2}
+C\int_{B(\mu)}|\bar{u}_{{\rho}}|^{2} \\
&\leq C\B(\int_{B(\f{\rho}{2})}|u-\bar{u}_{{\rho}}|^{2}\B)
  +
 C\f{\mu^{3}}{\rho^{3}}\B( \int_{B(\rho)}|u|^{2}\B)
\\
&\leq C \rho^{2s}
 \B(\int_{B^{\ast}(\rho)}y^{\wred{1-2s}}|  \nabla^{\ast}  u^{\ast}|^{2}dxdy\B) + C\f{\mu^{3}}{\rho^{3}}\B( \int_{B(\rho)}|u|^{2}\B),
 \end{aligned}
\ee
that is,
$$
\int_{B(\mu)}|\nabla u|^{2}dx    \leq C\rho^{2s}
 \B(\int_{B^{\ast}(\rho)}y^{\wred{1-2s}} | \nabla^{\ast} (\nabla u)^{\ast}|^{2}dxdy\B) + C\f{\mu^{3}}{\rho^{3}}\B( \int_{B(\rho)}|\nabla u|^{2}\B).
$$

Integrating in time on $(\wred{-\mu^{2}},\,0)$ this inequality, we obtain
$$\iint_{Q(\mu)}|\nabla u|^{2}dx    \leq C\rho^{2s}
 \B(\iint_{Q^{\ast}(\rho)}y^{\wred{1-2s}} | \nabla^{\ast} (\nabla u)^{\ast}|^{2}dxdy\B) + C\f{\mu^{3}}{\rho^{3}}\B( \iint_{Q(\rho)}|\nabla u|^{2}\B)\wred{,}$$
which leads to
$$
E_{\ast}(\nabla u;\mu)\leq(\f{\rho}{\mu})E^{\ast}_{\ast}(\nabla^{\ast} (\nabla u)^{\ast};\rho)+(\f{\mu}{\rho})^{2}E_{\ast}(\nabla u;\rho).
$$
 This achieves the proof of this lemma.
\end{proof}

\begin{proof}[Proof of Theorem \ref{the1.5}]
From Lemma \ref{lem3.3} and the iteration method as in \cite{[CKN]}, we know that
$$\limsup_{\mu\rightarrow0} E_{\ast}(\nabla u;\mu)\leq C\limsup_{\mu\rightarrow0} E^{\ast}_{\ast}(\nabla^{\ast} (\nabla u)^{\ast};\mu).$$
The famous  $\epsilon$-regularity  criterion \eqref{ckn} helps us to finish the proof.
\end{proof}

\begin{proof} [Proof of Theorem \ref{the1.4}]
For the case $s=0$, we complete the proof by the Caffarelli-Kohn-Nirenberg theorem in \cite{[CKN]}. For the other borderline case $s=1/2$, by the fact $\dot{H}^{\f32}(\mathbb{R}^{3})\hookrightarrow BMO$ and the Serrin class $ L^{2}(0,T;BMO)$ due to Kozono and Taniuchi \cite{[KT]}, we know there is no singular point in the weak solutions of the Navier-Stokes equations. Hence, we achieve the proof of two borderline cases. For the rest cases, from \eqref{CSK}, we derive from $  u\in L^{2}(0,T;\dot{H}^{s+1}(\mathbb{R}^{3}))$ with $0< s<\f12$ that
$$\int\int_{\mathbb R^4_+} y^{1-2s} |\nabla^{\ast} (\nabla u)^{\ast}|^2 (x,y,t)\, dx\, dy dt<+\infty.
$$
At this stage, the Vitali covering lemma used in \cite{[CKN]} together with Theorem \ref{the1.5} yields that $1-2 s$ dimension of potential singular points set of \wred{suitable} weak solutions
satisfying  $  u\in L^{2}(0,T;\dot{H}^{s+1}(\mathbb{R}^{3}))$ for $0<s <\f12$   is zero.  The process is standard, hence, we omit the detail here.
\wred{In summary}, the desired result is derived.
\end{proof}
The proof of Corollary \ref{coro1.6} is a consequence of  the following two lemmas. \begin{lemma}\label{lem4.2} Let $\nabla u\in L^{q}(0,T;L^{p}(\mathbb{R}^{3}))$ for $2\leq\f{2}{q}+\f{  3}{p}\leq\f52$ with $\f52-\f3p-\f{5}{2q}\geq0, 2<p<\f{54+12\sqrt{14}}{25},\wred{1<q\leq2}$. Then
$$u\in L^{\infty}(0,T;L^{2}(\mathbb{R}^{3}))\cap L^{2}(0,T;\dot{H}^{1+s}(\mathbb{R}^{3})),$$
where $\wred{0\leq} s=\f{\f52-\f3p-\f{5}{2q}}{1-\f1q}\wred{\leq \frac{1}{2}}$.
 \end{lemma}
\begin{proof}
The incompressible condition allow us to get
\be\label{3.1}
\langle u\cdot\nabla\wred{\Lambda^{s}u}, \Lambda^{s}u \rangle=0.\ee
Multiplying the Navier-Stokes equations with $\Lambda^{2s}u$, using the divergence free condition and \eqref{3.1}, we know that
$$
\f12\f{d}{dt}\|\Lambda^{s}u\|^2_{L^{2}(\mathbb{R}^{3})}+\|\Lambda^{s+1}u\|^2_{L^{2}(\mathbb{R}^{3})}=-\langle\Lambda^{s}(u\cdot\nabla u)-u\cdot\nabla(\Lambda^{s}u),\Lambda^{s}u\rangle.
$$
The H\"older inequality guarantees that
$$ |\langle\Lambda^{s}(u\cdot\nabla u)-u\cdot\nabla(\Lambda^{s}u),\Lambda^{s}u\rangle|
\leq\|\Lambda^{s}(u\cdot\nabla u)-u\cdot\nabla(\Lambda^{s}u)\|_{L^{2}(\mathbb{R}^{3})}\|\Lambda^{s}u\|_{L^{2}(\mathbb{R}^{3})}.$$
By means of   Kato-Ponce commutator inequality  \eqref{KPC}, we infer that
$$
\|\Lambda^{s}(u\cdot\nabla u)-u\cdot\nabla(\Lambda^{s}u)\|_{L^{2}(\mathbb{R}^{3})}\leq C
\|\nabla u\|_{L^{p}(\mathbb{R}^{3})}\|\Lambda^{s} u\|_{L^{\f{2p}{p-2}}(\mathbb{R}^{3})},{p>2}.
$$
Consequently, we arrive at
\be\label{3.2}
\f12\f{d}{dt}\|\Lambda^{s}u\|^2_{L^{2}(\mathbb{R}^{3})}+\|\Lambda^{s+1}u\|^2_{L^{2}(\mathbb{R}^{3})}\leq C
\|\nabla u\|_{L^{p}(\mathbb{R}^{3})}\|\Lambda^{s} u\|_{L^{\f{2p}{p-2}}(\mathbb{R}^{3})}\|\Lambda^{s}u\|_{L^{2}(\mathbb{R}^{3})}.
\ee
We conclude by the fractional   Gagliardo-Nirenberg inequality   (see e.g. \cite{[WXW],[WWY]} and references therein)   and   the Sobolev inequality that,
\be\label{3.3}\ba
\|\Lambda^{s} u\|_{L^{\f{2p}{p-2}}(\mathbb{R}^{3})}\leq C\|\nabla u\|^{\f{\f{3}{p}}{\f52-s-\f3p}}_{L^{p}(\mathbb{R}^{3})}\|u\|^{\f{ \f52-s-\f6p}{\f52-s-\f3p}}_{L^{\f{3}{\f32-s}}(\mathbb{R}^{3})}
\leq C\|\nabla u\|^{\f{\f{3}{p}}{\f52-s-\f3p}}_{L^{p}(\mathbb{R}^{3})}\|\Lambda^{s}u\|^{\f{ \f52-s-\f6p}{\f52-s-\f3p}}_{L^{2}(\mathbb{R}^{3})},
\ea\ee
where we require
$$0\leq\f{\f{3}{p}}{\f52-s-\f3p} \leq1, \f52-s-\f3p>0\quad\text{and}\quad s\leq\f{\f{3}{p}}{\f{5}{2}-s-\f3p}.$$
Indeed, in the light the definition of $s$ and \wred{$1<q\leq2$}, we observe that $\f{\f{3}{p}}{\f52-s-\f3p}\leq1$.  In addition, taking advantage of the the definition of $s$ again, we know that $\f52-s-\f3p>0$. Some straightforward computations yields that  $\f{9-\sqrt{56}}{6}<\f1p<\f{9+\sqrt{56}}{6}$ guarantees that $s\leq\f{\f{3}{p}}{\f{5}{2}-s-\f3p}$.

Inserting \eqref{3.3} into \eqref{3.2}, we have
\be\label{3.4}\ba
\f12\f{d}{dt}\|\Lambda^{s}u\|^2_{L^{2}(\mathbb{R}^{3})}
+\|\Lambda^{s+1}u\|^2_{L^{2}(\mathbb{R}^{3})}\leq& C
\|\nabla u\|^{\f{\f{3}{p}}{\f52-s-\f3p}+1}_{L^{p}(\mathbb{R}^{3})}\|\Lambda^{s}u\|^{\f{ \f52-s-\f6p}{\f52-s-\f3p}+1}_{L^{2}(\mathbb{R}^{3})}\\
\leq& C
\|\nabla u\|^{q}_{L^{p}(\mathbb{R}^{3})}\|\Lambda^{s}u\|^{\f{ \f52-s-\f6p}{\f52-s-\f3p}+1}_{L^{2}(\mathbb{R}^{3})}.
\ea\ee
Thanks to
$\f{ \f52-s-\f6p}{\f52-s-\f3p}\leq1$,  we derive from \eqref{3.4} and  $\nabla u\in L^{q}(0,T;L^{p}(\mathbb{R}^{3}))$ for $2\leq\f{2}{q}+\f{  3}{p}\leq\f52$ that $u\in L^{2}(0,T;H^{1+s}(\mathbb{R}^{3}))$.
\end{proof}
\begin{lemma}\label{lem4.3}  Let $u$ be a suitable weak solution belonging in $\nabla u\in L^{q}(0,T;L^{p}(\mathbb{R}^{3}))$ for $2\leq\f{2}{q}+\f{  3}{p}\leq\f52$ with $2-\f3p-\f1q\geq0,\f32<p<\f{12}{7},q\geq4$.  Then
$$u\in L^{\infty}(0,T;L^{2}(\mathbb{R}^{3}))\cap  L^{2}(0,T;\dot{H}^{1+s}(\mathbb{R}^{3})),$$
where $\wred{0\leq}s=\f{2-\f3p-\f1q}{\f2q}\wred{\leq \frac{1}{2}}$.
 \end{lemma}
 \begin{proof}
In view of the standard energy estimate, the integration by parts and the incompressible condition, we have
$$
\f12\f{d}{dt}\|\Lambda^{s}u\|^2_{L^{2}(\mathbb{R}^{3})}+\|\Lambda^{s+1}u\|^2_{L^{2}(\mathbb{R}^{3})}= \langle\Lambda^{s}(u\otimes u) ,\Lambda^{s+1}u\rangle.
$$
It follows from the H\"older inequality that
$$|\langle\Lambda^{s}(u\otimes u) ,\Lambda^{s+1}u\rangle|\leq \|\Lambda^{s}(u\otimes u)\|_{L^{2}(\mathbb{R}^{3})}\|\Lambda^{s+1}u\|_{L^{2}(\mathbb{R}^{3})}.$$
We deduce \wred{from} the  Kato-Ponce   product estimates \eqref{KPI}  and the Sobolev embedding that
$$\|\Lambda^{s}(u\otimes u)\|_{L^{2}(\mathbb{R}^{3})}\leq C\|\Lambda^{s}u\|_{L^{\f{6p}{5p-6}}(\mathbb{R}^{3})}\|u\|_{L^{\f{3p}{3-p}}(\mathbb{R}^{3})}\leq C\|\Lambda^{s}u\|_{L^{\f{6p}{5p-6}}(\mathbb{R}^{3})}\|\nabla u\|_{ L^{p}(\mathbb{R}^{3})}, \wred{\;\f65<p<3.}$$
Combining the above estimates together, we observe that
\begin{equation}\label{wu4.6}
  \f12\f{d}{dt}\|\Lambda^{s}u\|^2_{L^{2}(\mathbb{R}^{3})}
+\|\Lambda^{s+1}u\|^2_{L^{2}(\mathbb{R}^{3})}\leq C\|\Lambda^{s}u\|_{L^{\f{6p}{5p-6}}(\mathbb{R}^{3})}\|\nabla u\|_{ L^{p}(\mathbb{R}^{3})}\|\Lambda^{s+1}u\|_{L^{2}(\mathbb{R}^{3})}\wred{.}
\end{equation}
According to the fractional Gagliardo-Nirenberg inequality and the Sobolev inequality, we discover  that
\begin{equation}\label{wu4.7}
\|\Lambda^{s}u\|_{L^{\f{6p}{5p-6}}(\mathbb{R}^{3})}\leq C\|u\|_{L^{\f{3p}{3-p}}(\mathbb{R}^{3})}^{\f{2-\f3p}{s-\f32+\f3p}}
 \|\Lambda^{s+1}u\|_{L^{2}(\mathbb{R}^{3})}
^{\f{s-\f72+\f6p}{s-\f32+\f3p}}\leq C
 \|\nabla u\|_{ L^{p}(\mathbb{R}^{3})}^{\f{2-\f3p}{s-\f32+\f3p}}
 \|\Lambda^{s+1}u\|_{L^{2}(\mathbb{R}^{3})}
^{\f{s-\f72+\f6p}{s-\f32+\f3p}},
\end{equation}
where we need
$p\geq\f32,0\leq\f{s-\f72+\f6p}{s-\f32+\f3p}\leq1$, $s-\f32+\f3p>0$ and $ \f{s}{s+1}<\f{2-\f3p}{s-\f32+\f3p} $.\\
On one hand, we can examine
$\f{2-\f3p}{s-\f32+\f3p}\leq1$ via  $q\geq 4$ and $ \wred{3>p\geq\f32}$. On the other hand, direct calculation
ensures that $q>2, p>\f32$ yields that $s-\f32+\f3p>0$. Moreover, $p<\f{12}{7} $ means  $ \f{s}{s+1}<\f{2-\f3p}{s-\f32+\f3p}. $

\wred{Inserting \eqref{wu4.7} into \eqref{wu4.6}}, we find
$$\f12\f{d}{dt}\|\Lambda^{s}u\|^2_{L^{2}(\mathbb{R}^{3})}+\|\Lambda^{s+1}u\|^2_{L^{2}(\mathbb{R}^{3})}\leq C
\|\nabla u\|^{\f{s+\f12}{s-\f32+\f3p}}_{L^{p}(\mathbb{R}^{3})}\|\Lambda^{s+1}u\|^{\f{ s-\f72+\f6p}{s-\f32+\f3p}+1}_{L^{2}(\mathbb{R}^{3})},
$$
which implies that
$$
\|\Lambda^{s}u\|^2_{L^{\infty}(0,T;L^{2}(\mathbb{R}^{3}))}
+\|\Lambda^{s+1}u\|^2_{L^{2}(0,T;L^{2}(\mathbb{R}^{3}))}\leq \wred{C_0+}C\|\nabla u\|^{\f{s+\f12}{s-\f32+\f3p}}_{L^{\f{2s+1}{2-\f3p}}(0,T;L^{p}(\mathbb{R}^{3}))}\|\Lambda^{s+1}u\|^{\f{ s-\f72+\f6p}{s-\f32+\f3p}+1}_{L^{2}(0,T;L^{2}(\mathbb{R}^{3}))}.
$$
The proof of this lemma is completed.
 \end{proof}
\begin{proof}[Proof of Corollary \ref{coro1.6}]
Combining Lemma  \ref{lem4.2}, Lemma  \ref{lem4.3} and Theorem \ref{the1.4}, we immediately finish the proof of Corollary \ref{coro1.6}.
 \end{proof}
		\section*{Acknowledgement}
 The authors would like to express their sincere gratitude to Dr. Wei Wei at Northwest University, for the  discussion  involving the inequality \eqref{3.3} and \eqref{wu4.7}.
The research of Wang was partially supported by  the National Natural Science Foundation of China (No. 11971446 and  No. 11601492).
 The research of Wu was partially supported by the National Natural Science Foundation of China under grant No. 11771423.


\begin{thebibliography}{00}
\bibitem{[BW]} T. Barker and  W. Wang,
Estimates of the singular set for the Navier-Stokes equations with supercritical assumptions on the pressure. arXiv:2111.15444. 2022.

\bibitem{[Beirao da Veiga]}
H. Beirao da Veiga, A new regularity class for the Navier-Stokes equations in $\mathbb{R}^{n}$.
Chinese Annals of Mathematics. Series B, 16 (1995), 407--412.

 \bibitem{[BG]}
     L. C. Berselli and G. P. Galdi, Regularity criteria involving the pressure for the weak solutions of the Navier-Stokes
equations, Proc. Amer. Math. Soc., 130 (2002) 3585--3595.
\bibitem{[CKN]}
L. Caffarelli,  R. Kohn and L. Nirenberg,  Partial regularity
of suitable weak solutions of Navier-Stokes equation,   Comm. Pure. Appl. Math.,    35   (1982), 771--831.



\bibitem{[CS]}
L. Caffarelli and L. Silvestre, An extension problem related to the fractional Laplacian,   Comm. Partial Differential Equations.   32  (2007), 1245--1260.

 \bibitem{[CDM]}
M. Colombo,  C. De Lellis and A. Massaccesi, The generalized Caffarelli-Kohn-Nirenberg theorem for the hyperdissipative Navier-Stokes system. Comm. Pure Appl. Math. 73 (2020),  609--663.


 \bibitem{[ESS]} L. Escauriaza, G. Seregin and V. \v{S}ver\'{a}k, On $L^{\infty}L^{3}$ -solutions to the Navier-tokes equations and Backward uniqueness. Russian Math. Surveys.,  58 (2003),  211--250.

\bibitem{[Falconer]}
K. Falconer, Fractal Geometry: Mathematical
Foundations and Applications (New York: Wiley) 1990.

\bibitem{[GKT]}
 S. Gustafson,  K. Kang and T.  Tsai, Regularity criteria for suitable weak solutions of the Navier-Stokes equations near the boundary. J. Differential Equations, 226 (2006),   594--618.

 \bibitem{[HWZ]}
C. He, Y.  Wang and D.  Zhou, New $\epsilon$-regularity criteria of suitable weak solutions of the 3D Navier-Stokes equations at one scale. J. Nonlinear Sci. 29 (2019),   2681--2698.




 \bibitem{[KP1]}
T. Kato and  G. Ponce,  Commutator Estimates and Euler and Navier-Stokes Equations. Comm. Pure Appl. Math. 41(1988), 891--907.

\bibitem{[KT]}  H. Kozono and Y. Taniuchi,
Bilinear estimates in BMO
and the Navier-Stokes equations. Math. Z. 235 (2000), 173--194.


\bibitem{[Kukavica]}
I. Kukavica, The fractal dimension of the singular set for solutions of the Navier-Stokes system
 Nonlinearity., 22  (2009),  2889--2900.
\bibitem{[KP]} I.
Kukavica and Y. Pei, An estimate on the parabolic fractal dimension of the singular set for solutions of the Navier-Stokes system.   Nonlinearity., 25  (2012),  2775--2783.


































\bibitem{[LS]}O.  Ladyzenskaja and G.  Seregin,  On
partial regularity of suitable weak solutions to the three-dimensional Navier-Stokes equations,     J. Math.
Fluid Mech., 1    (1999),  356--387.

\bibitem{[Leray]}  J. Leray , Sur le mouvement d\'eun liquide visqueux
emplissant l\'espace,   Acta Math., 63  (1934),  193--248.

\bibitem{[Lin]}F. Lin,   A new proof of the Caffarelli-Kohn-Nirenberg
Theorem,    Comm. Pure Appl. Math., 51  (1998),  241--257.




\bibitem{[RWW]} W. Ren, Y. Wang and G. Wu,
Partial regularity of suitable weak solutions to the multi-dimensional generalized magnetohydrodynamics equations.  Commun. Contemp. Math.,  (2016).  1650018, 38 pp.


\bibitem{[RS1]}
J. Robinson  and W. Sadowski, Decay of weak solutions and the singular set of the three-dimensional
Navier-Stokes equations,   Nonlinearity., 20  (2007), 1185--1191.


\bibitem{[RS2]}
\bysame, Almost-everywhere uniqueness of Lagrangian trajectories for suitable
weak solutions of the three-dimensional Navier-Stokes equations,  Nonlinearity.,  22  (2009) 2093--2099.


\bibitem{[RS3]}
\bysame,
On the Dimension of the Singular Set of Solutions to the Navier-Stokes Equations, Comm. Math.
Phys.,  309   (2012),  497--506.








\bibitem{[Scheffer]}
V. Scheffer,  Turbulence and Hausdorff dimension,
in Turbulence and the Navier-Stokes Equations,
 Lecture Notes in Math.,  Springer-Verlag,   565  (1976),    94--112.
\bibitem{[Scheffer1]}
\bysame,   Partial regularity of solutions to the Navier-Stokes
equations,       Pacific J. Math., 66  (1976),  535--552.

\bibitem{[Scheffer2]}\bysame,  Hausdorff measure and the Navier-Stokes equations,   Comm. Math. Phys.,  55   (1977),  97--112.

\bibitem{[Scheffer3]}
\bysame,  The Navier-Stokes equations in space dimension
four,       Comm. Math. Phys., 61  (1978),  41--68.
 \bibitem{[Struwe]} M. Struwe, On partial regularity results for the Navier-Stokes equations, Comm. Pure Appl. Math., 41 (1988),  437--458.

\bibitem{[Serrin]} J. Serrin, On the interior regularity of weak solutions of the Navier-Stokes equations, Arch. Ration. Mech. Anal.,  9  (1962) 187--195.






\bibitem{[Taniuchi]}
Y. Taniuchi, On generalized energy equality of the Navier-Stokes equations.
Manuscripta Math. 94 (1997), 365--384.

\bibitem{[TY]}
 L. Tang and Y. Yu, Partial regularity of suitable weak solutions to the fractional Navier-Stokes equations.   Comm. Math. Phys.,  334  (2015), 1455--1482.




\bibitem{[WZ]}
W. Wang and Z. Zhang,   On the interior regularity criteria and the number of singular points to the Navier-Stokes equations.   J. Anal. Math., 123  (2014), 139--170.

 \bibitem{[WW]}Y. Wang and G. Wu,
A unified proof on the partial regularity for  suitable weak solutions of non-stationary and  stationary  Navier-Stokes equations.   J. Differential Equations., 256  (2014),  1224--1249.



\bibitem{[WXW]} Y.  Wang, X. Mei  and W. Wei. Gagliardo-Nirenberg inequality in anisotropic Lebesgue spaces and energy equality in the Navier-Stokes equations.  arXiv:2204.07479. 2022.

\bibitem{[WWY]}   W. Wei, Y. Wang and Y. Ye, Gagliardo-Nirenberg inequalities in Lorentz type spaces and energy equality for the Navier-Stokes system.
arXiv:2106.11212. 2021.





\bibitem{[WWZ]}
Y. Wang, G. Wu and D. Zhou, A regularity criterion at one scale without pressure for suitable weak solutions to the Navier-Stokes equations. J. Differential Equations 267 (2019),   4673--4704.
\bibitem{[WY]}
Y. Wang and M. Yang, Improved bounds for box dimensions of potential singular points to the Navier-Stokes equations. Nonlinearity, 32 (2019),   4817--4833.




\bibitem{[Zhou1]} Y. Zhou, Regularity criteria in terms of pressure for the 3-D Navier-Stokes equations in a generic domain, Math.
Ann., 328 (2004), 173--192.
\bibitem{[Zhou2]} Y. Zhou, On regularity criteria in terms of pressure for the Navier-Stokes equations in $\mathbb{R}^{3}$, Proc. Amer. Math.
Soc., 134 (2005), 149--156.
\bibitem{[Yang]}
R. Yang, On higher order extensions for the fractional Laplacian. Preprint, 2013.
arXiv:1302.4413

\end{thebibliography}
	\end{document}